\documentclass[10pt]{amsart}

\usepackage{amssymb,amsmath,amsthm}
\usepackage{tikz}
\usepackage{tikz-cd}
\usepackage[arc,all]{xy}
\usepackage{eucal}

\usepackage{color}

\usepackage{comment}

\usepackage[justification=centering]{caption}
\usepackage{longtable, makecell, multirow, tabularx}

\theoremstyle{plain}

\newtheorem{lem}[subsection]{Lemma}
\newtheorem{thm}[subsection]{Theorem}
\newtheorem{prop}[subsection]{Proposition}
\newtheorem{notation}[subsection]{Notation}

\theoremstyle{definition}

\theoremstyle{remark}
\newtheorem{rem}[subsection]{Remark}

\numberwithin{equation}{section}

\title[Higher Adams differentials and algebraic Novikov differentials]{A correspondence between higher Adams differentials and higher algebraic Novikov differentials at odd primes}

\author[X. Wang]{Xiangjun Wang}
\author[Y. Zhang]{Yu Zhang$^*$}

\address{Department of Mathematics, Nankai University, No.94 Weijin Road, Tianjin 300071, P. R. China}
\email{xjwang@nankai.edu.cn}

\address{Department of Mathematics, Nankai University, No.94 Weijin Road, Tianjin 300071, P. R. China}
\email{zhang.4841@buckeyemail.osu.edu}



\thanks{$*$ Corresponding author}

\begin{document}

\maketitle

\begin{abstract}
    This paper studies the higher differentials of the classical Adams spectral sequence at odd primes.  In particular, we follow the ``cofiber of $\tau$ philosophy'' of Gheorghe, Isaksen, Wang, and Xu to show that higher Adams differentials agree with their corresponding higher algebraic Novikov differentials in a certain range. 
\end{abstract}

\section{Introduction}
\label{sec:intro}

The computation of the stable homotopy groups of the sphere $\pi_*(S^0)$ is one of the most important problems in homotopy theory.  In recent years, a major breakthrough in this area is the work of Isaksen-Wang-Xu \cite{Isaksen_Wang_Xu_more_stable_stems}, which successfully extended the computation of the $2$-primary components of $\pi_*(S^0)$ to dimension 90.  Their computation is based on the 
``cofiber of $\tau$ philosophy'' developed in Gheorghe-Wang-Xu \cite{Gheorghe_Wang_Xu}.  The key insight of \cite{Gheorghe_Wang_Xu, Isaksen_Wang_Xu_more_stable_stems} is that one can compute difficult higher Adams differentials from higher algebraic Novikov differentials
using motivic homotopy theory. 
Although the method developed in \cite{Gheorghe_Wang_Xu} holds for arbitrary primes, it has only been extensively exploited for $2$-primary computations.  In this paper, we apply the 
``cofiber of $\tau$ philosophy'' to odd-primary computations.

\subsection*{Some classical spectral sequences}

Throughout this paper, we let $p$ denote a fixed odd prime.
The Adams spectral sequence (ASS) and the Adams-Novikov spectral sequence (ANSS) are two of the most powerful tools to compute the $p$-primary component of $\pi_*(S^0)$. 

The $E_2$-page of the ASS is $Ext_{\mathcal{A}_{*}}^{*,*} (\mathbb{F}_p, \mathbb{F}_p)$, where 
$\mathcal{A}_{*}$ 
is the dual mod $p$ Steenrod algebra.  We can write 
$\mathcal{A}_{*} = \mathbb{F}_p [t_1, t_2, \cdots] \otimes E [\tau_0, \tau_1, \tau_2, \cdots]$, where $\mathbb{F}_p [t_1, t_2, \cdots]$ is a polynomial algebra with coefficients in $\mathbb{F}_p$, and $E [\tau_0, \tau_1, \tau_2, \cdots]$ is an exterior algebra with coefficients in $\mathbb{F}_p$.  The differentials of the ASS have the form $d^{Adams}_r: E^{s,t}_r (S) \rightarrow E^{s+r,t+r-1}_r (S)$, $r \geq 2$.

The $E_2$-page of the ANSS is $Ext_{BP_*BP}^{*,*}(BP_*, BP_*)$, where $BP$ denotes the Brown-Peterson spectrum. We have $BP_*:= \pi_*(BP) = \mathbb{Z}_{(p)}[v_1, v_2, \cdots]$, and that $BP_*BP = BP_* [t_1,t_2,\cdots]$, where $\mathbb{Z}_{(p)}$ denotes the integers localized at $p$.  The inner degrees of the generators are $|v_n| = |t_n| = 2(p^n - 1)$.

The algebraic Novikov spectral sequence (algNSS) \cite{Miller_algANSS,Novikov_methods} converges to the Adams-Novikov $E_2$-page.  The $E_2$-page of the algNSS is $Ext^{s,t}_{P_*} (\mathbb{F}_p, I^k / I^{k+1})$, where $I$ denotes the ideal $(p, v_1, v_2, \cdots) \subset BP_*$, and $P_* = BP_* BP/I = \mathbb{F}_p [t_1, t_2, \cdots]$ is the $\mathbb{F}_{p}$-coefficient polynomial algebra. The differentials have the form $d_r^{alg}: \Bar{E}_r^{s,t,k} \to \Bar{E}_r^{s+1,t,k+r-1}$, $r \geq 2$.
 
\begin{rem}
Here, we have re-indexed the pages of the algNSS to align with the notations in \cite{Gheorghe_Wang_Xu, Isaksen_Wang_Xu_more_stable_stems}.
\end{rem}

The $E_2$-page of the Adams spectral sequence can also be computed via a spectral sequence, called the Cartan-Eilenberg spectral sequence (CESS) \cite{Cartan_Eilenberg_homological_algebra, Ravenel_stable_homotopy_groups}.  For odd primes $p$, the $E_2$-page of the CESS coincides with the $E_2$-page of the algNSS up to degree shifting \cite{Miller_algANSS}.

These four spectral sequences fit into the following square \cite{Miller_relations}.
\begin{align}
\xymatrix@C=5em{
Ext_{P_*}^{s,t} (\mathbb{F}_{p}, I^k / I^{k+1}) \ar@{=>}[r]^{CESS} \ar@{=>}[d]_{algNSS} & Ext_{\mathcal{A}_*}^{s+k,t+k} (\mathbb{F}_p , \mathbb{F}_p) \ar@{=>}[d]^{ASS}  \\ 
Ext^{s,t}_{BP_*BP}(BP_*, BP_*) \ar@{=>}[r]^{\quad \quad \quad   ANSS}  & \pi_{t-s} (\hat{S}^0)  
}
\end{align}

The Adams differentials $d^{Adams}_r$ and the Adams-Novikov differentials $d^{AN}_r$ are difficult to compute in general.  Such computations might require an understanding of complicated geometric behaviors.  On the other hand, the algebraic Novikov differentials $d^{alg}_r$ are much easier to compute.  This is because the entire construction of the algNSS is purely algebraic. 

It is long been observed that $d^{alg}_r$ and $d^{Adams}_r$ are closely related.  It is desirable if we could determine $d^{Adams}_r$ differentials based on $d^{alg}_r$ computations.  The following result realizes the idea for $r=2$, while the problem for general  $r>2$ remains open.

\begin{thm}[Novikov \cite{Novikov_methods}, Andrews-Miller \cite{Andrews_Miller,Miller_relations}]
\label{thm: miller transition}
Suppose $z$ is an element in $Ext_{\mathcal{A}_*}^{s+k,t+k} (\mathbb{F}_p , \mathbb{F}_p)$
which is detected in the CESS by $x \in Ext_{P_*}^{s,t} (\mathbb{F}_{p}, I^k / I^{k+1})$. Then, $d^{Adams}_2 (z)$ is detected by $d^{alg}_2 (x) \in Ext^{s+1,t}_{P_*}(\mathbb{F}_{p}, I^{k+1} / I^{k+2})$.
\end{thm}

\begin{rem}
    Indeed, one can use this approach to find new nontrivial secondary Adams differentials.  See \cite{Wang_Wang_Zhang} for some practical examples.
\end{rem}

\subsection*{New developments using motivic homotopy theory}

For general $r \geq 2$, Gheorghe-Wang-Xu \cite{Gheorghe_Wang_Xu} developed the ``cofiber of $\tau$ philosophy'' to compare $d^{alg}_r$ with $d^{Adams}_r$.  
One key feature of their approach is the usage of motivic homotopy theory, which originated from the work of Morel and Veovodsky \cite{Morel_Voevodsky}.  Motivic homotopy theory behaves very similarly to the classical homotopy theory.  For example, there are motivic dual Steenrod algebras \cite{Voevodsky_reduced_power, Voevodsky_motivic_EM} and motivic Adams spectral sequences (mASS) \cite{Duggerz_Isaksen_motivic_Adams, Hu_Kriz_Ormsby_convergence}.

The computational technique of \cite{Gheorghe_Wang_Xu} can be illustrated in the following diagram.
\begin{align}
\label{diagram: GWX diagram}
\xymatrix{
 Ext^{*,*,*}_{\mathcal{A}_{*,*}^{\mathbb{C}}} (\mathbb{F}_p [\tau], \mathbb{F}_p)   \ar@{=>}[d]^{mASS}  & Ext^{*,*,*}_{\mathcal{A}_{*,*}^{\mathbb{C}}} (\mathbb{F}_p [\tau], \mathbb{F}_p[\tau])   \ar@{=>}[d]^{mASS} \ar[r] \ar[l] & Ext^{*,*}_{\mathcal{A}_*} (\mathbb{F}_p, \mathbb{F}_p)  \ar@{=>}[d]^{ASS} \\
 \pi_{*,*} (\widehat{S^{0, 0}}/ \tau)   & \pi_{*,*} (\widehat{S^{0, 0}}) \ar[r] \ar[l]   & \pi_*(\widehat{S^{0}})
} 
\end{align}
They considered a map $\tau: \widehat{S^{0, -1}} \to \widehat{S^{0, 0}}$ between $p$-completed motivic sphere spectra.  We let $\widehat{S^{0, 0}}/ \tau$ denote the cofiber of the map $\tau$. 
In diagram \eqref{diagram: GWX diagram}, the left column is the mASS for $\widehat{S^{0, 0}}/ \tau$, the middle column is the mASS for $\widehat{S^{0, 0}}$, and the right column is the classical ASS.  
The top horizontal maps are maps of spectral sequences.  Gheorghe-Wang-Xu \cite{Gheorghe_Wang_Xu} proved the left column is isomorphic to the algNSS for any prime $p$.  Hence diagram \eqref{diagram: GWX diagram}
provides a zig-zag diagram to compare higher Adams differentials with their corresponding higher algebraic Novikov differentials.

For $p=2$, Isaksen-Wang-Xu \cite{Isaksen_Wang_Xu_more_stable_stems} used the strategy above to determine $d^{Adams}_r$ based on computer generated $d^{alg}_r$ data.  This enabled them to successfully extend the computation of 2-primary stable homotopy groups from around 60 stems to around 90 stems.

\subsection*{Our main results}

In this paper, we follow the ``cofiber of $\tau$ philosophy'' but focus on the case when $p$ is an odd prime.   
For odd prime $p$, the motivic Steenrod algebra has a simpler form and the diagram 
\eqref{diagram: GWX diagram} presents different features. 
Let us explain our notations and then state our main results.

\begin{notation}
Let $x$ be an element in the $E_2$-page of a certain spectral sequence.  We let $[x]_k$ denote the homology class of $x$ on the $E_k$-page, where $[x]_2$ is understood to be $x$.    We say $x$ can be lifted to the $E_r$-page if $[x]_r$ is well defined, i.e., if $d_i ([x]_i) = 0$ for each $2 \leq i < r$.  
\end{notation}

\begin{thm}
\label{thm: alg eq adams}
 Let $p$ be an odd prime. Let $z \in Ext_{\mathcal{A}_*}^{s+k,t+k} (\mathbb{F}_p , \mathbb{F}_p)$ be an element in the $E_2$-page of the ASS which is detected by $x \in Ext_{P_*}^{s,t} (\mathbb{F}_{p}, I^k / I^{k+1})$ with $s< 2p-2$.  Let $r$ be an integer such that $r + k \leq 2p - 2$ and $z$ can be lifted to the $E_r$-page.  Then we can write $d^{Adams}_r ([z]_r) = [w]_r$, $d^{alg}_r ([x]_r) = [y]_r$, where $w \in Ext_{\mathcal{A}_*}^{s+k+r,t+k+r-1} (\mathbb{F}_p , \mathbb{F}_p)$ can be detected by $y \in Ext^{s+1,t}_{P_*}(\mathbb{F}_{p}, I^{k+r-1} / I^{k+r})$.
\end{thm}

\begin{rem}
In the motivic context, we say $z \in Ext_{\mathcal{A}_*}(\mathbb{F}_p, \mathbb{F}_p)$ is detected by $x \in Ext_{P_*}(\mathbb{F}_p, I^k / I^{k+1})$ if and only if $z$ lifts to $\tilde{z} \in Ext_{\mathcal{A}_{*,*}^{\mathbb{C}}}(\mathbb{F}_p[\tau], \mathbb{F}_p[\tau])$ which maps to $x \in Ext_{\mathcal{A}_{*,*}^{\mathbb{C}}}(\mathbb{F}_p[\tau], \mathbb{F}_p) \cong Ext_{P_*}(\mathbb{F}_p, I^k / I^{k+1})$, as shown in diagram \eqref{diagram: GWX diagram}.
\end{rem}

Theorem \ref{thm: alg eq adams} shows, in a certain range, the higher Adams differentials agree with their corresponding higher algebraic Novikov differentials.  It is worth pointing out that the result could fail for $r \geq 2p - 1 - k$.   For example, let $p \geq 5$.  The $k$-value for $b_n \in Ext_{\mathcal{A}_*}^{*,*} (\mathbb{F}_p , \mathbb{F}_p)$ is 0.  It is proved that $b_n$ has nontrivial Adams differential \cite{Ravenel_odd_arf} 
\begin{equation}
    d_{2p-1}^{Adams} (b_n) = h_0 b_{n-1}^p.
\end{equation}
However, the corresponding algebraic Novikov differential is trivial
\begin{equation}
    d_{2p-1}^{alg} (b_n) = 0.
\end{equation}

We also have the following result comparing the length of nontrivial algebraic Novikov differentials and Adams differentials.

\begin{thm}
\label{thm: nontrivial alg d then nontrivial adams d}
Let $p$ be an odd prime. Let $z \in Ext_{\mathcal{A}_*}^{s+k,t+k} (\mathbb{F}_p , \mathbb{F}_p)$ be an element detected  by $x \in Ext_{P_*}^{s,t} (\mathbb{F}_{p}, I^k / I^{k+1})$ with $s < 2p - 2$.  If $x$ is not a permanent cycle in the algNSS, then $z$ is not a permanent cycle in the ASS. 
Moreover, let $r$ (resp. $r'$) be the largest number $n$ such that $x$ (resp. $z$) can be lifted to the $E_{n}$-page, then $r \geq r'$.
\end{thm}

\subsection*{Organization of the paper} 

In Section \ref{sec: motivic}, we review some basic results in motivic homotopy theory as well as the ``cofiber of $\tau$'' method of Gheorghe, Isaksen, Wang, and Xu to compare higher algebraic Novikov differentials with higher Adams differentials. In Section \ref{sec: proof}, we give the proofs for Theorem \ref{thm: alg eq adams} and Theorem \ref{thm: nontrivial alg d then nontrivial adams d}.

\subsection*{Acknowledgments}
We would like to thank the anonymous referee for their constructive suggestions and detailed comments in revising this paper.
The authors are supported by the National Natural Science Foundation of China (No. 12271183).  The second named author is also supported by the National Natural Science Foundation of China (No. 12001474; 12261091).

\section{Comparing $d_r^{Adams}$ with $d_r^{alg}$}
\label{sec: motivic}

Motivic homotopy theory originated out of the work of Voevodsky and Morel \cite{Morel_Voevodsky, Voevodsky_reduced_power, Voevodsky_motivic_EM}. Motivic homotopy theory can be viewed as a successful application of abstract homotopy theory to algebraic geometry and number theory.  From categorical and computational perspectives, the motivic stable homotopy category behaves very similarly to the classical stable homotopy category.  For example, there are motivic spheres, motivic homotopy and homology groups, motivic Eilenberg-MacLane spectra, and motivic Steenrod algebras analogous to the classical ones.
For readers not familiar with motivic homotopy theory, \cite{Duggerz_Isaksen_motivic_Adams, Morel_intro_A1, Voevodsky_lectures} are some helpful references.

In this paper, we choose to work over the field $\mathbb{C}$ of complex numbers and assume $p$ is an odd prime.  Under this setting, we have explicit formulas for the mod $p$ motivic Eilenberg-MacLane spectra $H\mathbb{F}_p^{\text{mot}}$  and 
motivic dual Steenrod algebra $\mathcal{A}^{\mathbb{C}}_{*,*}$.

\begin{prop}[\cite{Voevodsky_reduced_power,Voevodsky_motivic_EM}]
\label{prop: motivic dual steenrod algebra}
Let $p$ be an odd prime.  We have 
$$H\mathbb{F}_{p \quad *,*}^{\text{mot}} = \mathbb{F}_p [\tau],$$
where $\tau$ has bi-degree $(0, -1)$, and that 
$$\mathcal{A}^{\mathbb{C}}_{*,*} = \mathbb{F}_p [\tau] \otimes \mathbb{F}_p [t_1, t_2, \cdots] \otimes E[\tau_0, \tau_1, \cdots],$$
where $\tau$ has bi-degree $(0, -1)$, $t_i$ has bi-degree $(2(p^i-1), p^i-1)$, and $\tau_i$ has bi-degree $(2(p^i-1)+1, p^i-1)$.  
\end{prop}

Under this setting, $\mathcal{A}^{\mathbb{C}}_{*,*} \cong \mathbb{F}_p [\tau] \otimes_{\mathbb{F}_p} \mathcal{A}_*$ is just the classical dual Steenrod algebra tensoring the new coefficient.

There is a motivic analog of the classical Adams spectral sequence called the motivic Adams spectral sequence (mASS) (see \cite{Duggerz_Isaksen_motivic_Adams,Hu_Kriz_Ormsby_convergence}).   

\begin{prop}
\label{prop: motivic adams ss}
There is a motivic Adams spectral sequence which converges to the bi-graded homotopy groups of the $H\mathbb{F}_p^{\text{mot}}$-completed motivic sphere $\widehat{S^{0, 0}}$.  The mASS has $E_2$-page
$$E_2^{s,t,u} (S) = Ext^{s,t,u}_{\mathcal{A}_{*,*}^{\mathbb{C}}} (\mathbb{F}_p [\tau], \mathbb{F}_p[\tau]) \Longrightarrow \pi_{t-s,u} (\widehat{S^{0, 0}}),$$
and differentials
$$d_r: E^{s,t,u}_r (S) \rightarrow E^{s+r,t+r-1,u}_r (S).$$
For $x \in E^{s,t,u}_r (S)$, we call $s$ its homological degree, $t$ the inner degree, and $u$ the motivic weight.  
\end{prop}

There is a topological realization functor $Re: SH^{\text{mot}} \to SH$ (see \cite{Dugger_Isaksen_motivic_realization, Morel_Voevodsky}) from the motivic stable homotopy category $SH^{\text{mot}}$ to the classical stable homotopy category $SH$. This functor maps the motivic sphere $S^{a,b}$ to classical sphere $S^a$ and maps the motivic Eilenberg-MacLane spectrum $H\mathbb{F}_p^{\text{mot}}$ to the classical Eilenberg-MacLane spectrum $H\mathbb{F}_p$.
For the other direction, there is also a constant embedding functor $C: SH \to SH^{\text{mot}}$. We have $Re \circ C = id$.

The functor $Re$ induces a map $\phi$ from the mASS of the motivic sphere to the classical ASS of the classical sphere.
\begin{align}
\label{diagram: mASS and ASS}
\xymatrix{
 E^{s,t,u}_{r} (S)  \ar[d]_{d_r^{mAdams}} \ar[r]^{\phi_r}  & E^{s,t}_{r} (S) \ar[d]^{d_r^{Adams}} \\
 E^{s+r,t+r-1,u}_{r} (S)  \ar[r]^{\phi_r}   & E^{s+r,t+r-1}_{r} (S) \\
} 
\end{align}
The effect of $\phi_r$ is inverting $\tau$ (see \cite{Dugger_Isaksen_motivic_realization, Duggerz_Isaksen_motivic_Adams, Isaksen_stable_stems}).

\begin{prop}[\cite{Duggerz_Isaksen_motivic_Adams,Isaksen_stable_stems}]
\label{prop: motivic e2 page for sphere}
Let $p$ be an odd prime. There is an isomorphism
$$Ext^{s,t,*}_{\mathcal{A}_{*,*}^{\mathbb{C}}} (\mathbb{F}_p [\tau], \mathbb{F}_p[\tau])  \cong \mathbb{F}_p [\tau] \otimes_{\mathbb{F}_p} Ext^{s,t}_{\mathcal{A}_*} (\mathbb{F}_p, \mathbb{F}_p).$$
Moreover, after inverting $\tau$, the mASS of the motivic sphere becomes isomorphic to the classical ASS tensored over $\mathbb{F}_p$ with $\mathbb{F}_p [\tau, \tau^{-1}]$.
\end{prop}

The element $\tau$ can be lifted to a map $\tau: \widehat{S^{0, -1}} \to \widehat{S^{0, 0}}$ between $H\mathbb{F}_{p}^{\text{mot}}$-completed motivic spectra.   We denote the associated cofiber sequence as 
\begin{equation*}
    \widehat{S^{0, -1}} \xrightarrow{\tau} \widehat{S^{0, 0}} \xrightarrow{i} \widehat{S^{0, 0}}/ \tau.
\end{equation*}

The mASS for $\widehat{S^{0, 0}}/ \tau$ has $E_2$-page
$E_2^{s,t,u} (C \tau) = Ext^{s,t,u}_{\mathcal{A}_{*,*}^{\mathbb{C}}} (\mathbb{F}_p [\tau], \mathbb{F}_p)$. 
To avoid potential confusions with the differentials in the mASS for $\widehat{S^{0, 0}}$, we denote the differentials in the mASS for $\widehat{S^{0, 0}}/ \tau$ as $\delta_r^{mAdams}$.
The map $i$ induces a map $\psi$ from the mASS for $\widehat{S^{0, 0}}$ to the mASS for $\widehat{S^{0, 0}}/ \tau$.
\begin{align}
\label{diagram: mASS and mASS}
\xymatrix{
 E^{s,t,u}_{r} (C \tau)  \ar[d]_{\delta_r^{mAdams}}   & E^{s,t,u}_{r} (S) \ar[d]^{d_r^{mAdams}} \ar[l]_{\psi_r} \\
 E^{s+r,t+r-1,u}_{r} (C \tau)   & E^{s+r,t+r-1,u}_{r} (S) \ar[l]_{\psi_r}  \\
} 
\end{align}
The effect of $\psi$ is just sending $\tau$ to $0$ (see \cite{Gheorghe_Wang_Xu, Isaksen_Wang_Xu_more_stable_stems}). 
 
Finally, Gheorghe-Wang-Xu \cite{Gheorghe_Wang_Xu} proved there is an isomorphism $\kappa$ between the mASS for $\widehat{S^{0, 0}}/ \tau$ and the regraded algebraic Novikov spectral sequence.
\begin{align}
\label{diagram: alg and mASS}
\xymatrix{
 \Bar{E}_r^{s+2u-t,2u,t-2u} \ar[d]_{d_r^{alg}} \ar[r]^{\kappa_r}_{\cong}  & E^{s,t,u}_{r} (C \tau) \ar[d]^{\delta_r^{Adams}} \\
 \Bar{E}_r^{s+2u-t+1,2u,t-2u+r-1}  \ar[r]^{\kappa_r}_{\cong}    & E^{s+r,t+r-1,u}_{r} (C \tau) \\
} 
\end{align}

We can summarize these three comparison maps in the following diagram.
\begin{align}
\label{diagram: three comparisons}
\xymatrixcolsep{15pt}
\xymatrix{
 \Bar{E}_r^{s+2u-t,2u,t-2u} \ar[d]_{d_r^{alg}} \ar[r]^{\quad \kappa_r}_{\quad \cong}  & E^{s,t,u}_{r} (C \tau) \ar[d]_{\delta_r^{mAdams}} & E^{s,t,u}_{r} (S) \ar[l]_{\quad \psi_r} \ar[d]_{d_r^{mAdams}} \ar[r]^{\phi_r}  & E^{s,t}_{r} (S) \ar[d]_{d_r^{Adams}} \\
 \Bar{E}_r^{s+2u-t+1,2u,t-2u+r-1}  \ar[r]^{\quad \kappa_r}_{\quad \cong}    & E^{s+r,t+r-1,u}_{r} (C \tau)  & E^{s+r,t+r-1,u}_{r} (S) \ar[l]_{\quad \psi_r} \ar[r]^{\phi_r}   & E^{s+r,t+r-1}_{r} (S) \\
} 
\end{align}

The diagram \eqref{diagram: three comparisons} provides a zig-zag way  to compare higher Adams differentials $d_r^{Adams}$ with their corresponding higher algebraic Novikov differentials $d_r^{alg}$.

\section{Proof of Theorems \ref{thm: alg eq adams}, \ref{thm: nontrivial alg d then nontrivial adams d}}
\label{sec: proof}

We discuss several lemmas before we prove our main results.

\begin{lem}
\label{lem: direct sum decompose Adams E2}
Let $p$ be an odd prime, denote $q = 2(p - 1)$. Given $s, t \geq 0$, we denote $C_{s,t} = \{i \in \mathbb{Z} | i \equiv t ~ ~ \text{mod} ~ q, ~ 0 \leq i \leq s,t \}$. Then we have the following direct sum decomposition of the classical Adams $E_2$-terms.  
$$Ext_{\mathcal{A}_*}^{s,t} (\mathbb{F}_p , \mathbb{F}_p) \cong \bigoplus_{i \in C_{s,t}} Ext_{P_*}^{s-i,t-i} (\mathbb{F}_{p}, I^i / I^{i+1})$$
\end{lem}

\begin{proof}
For odd prime $p$, the Cartan-Eilenberg spectral sequence
collapses from $E_2$-page with no nontrivial extensions \cite[Theorem 4.4.3]{Ravenel_stable_homotopy_groups}.
Hence we have 
$$Ext_{\mathcal{A}_*}^{s,t} (\mathbb{F}_p, \mathbb{F}_p)  \cong \bigoplus_{i \in \mathbb{Z}} Ext_{P_*}^{s-i,t-i} (\mathbb{F}_{p}, I^i / I^{i+1}).$$ 
Note the inner degrees  $|v_n| = |t_n| = 2(p^n - 1)$ are all multiples of $q$.  In order for $Ext^{s-i,t-i}_{P_*}(\mathbb{F}_p, I^i / I^{i+1})$
to be nontrivial, we need $i, s-i, t-i \geq 0$, and that $t-i \equiv 0 ~ ~ \text{mod} ~ q$.  Hence $i$ needs to be in the set $C_{s,t}$.

\end{proof}

\begin{rem}
It is worth pointing out that the Adams differential $d_2^{Adams}$ may not respect this decomposition.
\end{rem}

\begin{notation}
Let $z \in Ext_{\mathcal{A}_*}^{s,t} (\mathbb{F}_p , \mathbb{F}_p)$ be an element in the Adams $E_2$-page. We let $\tilde{z}$ denote the element 
$1 \otimes z \in \mathbb{F}_p [\tau] \otimes_{\mathbb{F}_p} Ext^{s,t}_{\mathcal{A}_*} (\mathbb{F}_p, \mathbb{F}_p) \cong Ext^{s,t,*}_{\mathcal{A}_{*,*}^{\mathbb{C}}} (\mathbb{F}_p [\tau], \mathbb{F}_p[\tau])$ in the $E_2$-page of the mASS of the motivic sphere.
\end{notation}

Note we have $\phi_2 (\tilde{z}) = z$. If $z$ and $\tilde{z}$ can both be lifted to the $E_r$-pages of the respected spectral sequences for some $r$.  Then $\phi_r ([\tilde{z}]_r) = [z]_r$.

\begin{lem}
\label{lem: compute motivic weight}
Let $z$ be an element in $Ext_{\mathcal{A}_*}^{s,t} (\mathbb{F}_p , \mathbb{F}_p)$ which is detected by $x \in Ext^{s-k,t-k}_{P_*}(\mathbb{F}_p, I^k / I^{k+1})$. 
Then $\tilde{z} \in Ext^{s,t,*}_{\mathcal{A}_{*,*}^{\mathbb{C}}} (\mathbb{F}_p [\tau], \mathbb{F}_p[\tau])$ has motivic weight $\frac{t-k}{2}$.
\end{lem}

\begin{proof}
For notation simplicity, we let $t(-)$ denote the inner degree and let $u(-)$ denote the motivic weight of an element in $Ext^{*,*,*}_{\mathcal{A}_{*,*}^{\mathbb{C}}} (\mathbb{F}_p [\tau], \mathbb{F}_p[\tau])$. By Proposition \ref{prop: motivic dual steenrod algebra}, we have $u(\tilde{t}_i) = \frac{1}{2} t(\tilde{t}_i)$ for $i \geq 1$, and $u(\tilde{\tau}_i) = \frac{1}{2} t(\tilde{\tau}_i) - \frac{1}{2}$ for $i \geq 0$.

Since $z$ is detected by  $x \in Ext^{s-k,t-k}_{P_*}(\mathbb{F}_p, I^k / I^{k+1})$, the number of $\tau_i$'s in the expression of $z$ is just $k$.  Hence we conclude $u(\tilde{z}) = \frac{1}{2} t(\tilde{z}) - \frac{1}{2} k = \frac{t-k}{2}$.
\end{proof}

\begin{lem}
\label{lem: lifting to Er}
Let $z \in Ext_{\mathcal{A}_*}^{s+k,t+k} (\mathbb{F}_p , \mathbb{F}_p)$ be an element which is detected by $x \in Ext_{P_*}^{s,t} (\mathbb{F}_{p}, I^k / I^{k+1})$.  (i) If $\tilde{z}$ can be lifted to the $E_r$-page for some $r \geq 2$, then $z$ can also be lifted to the $E_r$-page. (ii) Further assume $s < 2p - 2$.  Then $z$ can be lifted to the $E_r$-page implies $\tilde{z}$ can also be lifted to the $E_r$-page.
\end{lem}

\begin{proof}
(i) By assumption, $d_i^{mAdams} ([\tilde{z}]_i) = 0$ for each $2 \leq i < r$. One can inductively show that $d_i^{Adams} ([z]_i) = 0$ (hence $[z]_{i+1}$ is well defined) for each $2 \leq i < r$ by commutativity of diagram \eqref{diagram: mASS and ASS}.   

(ii) Assume, for the sake of contradiction, that there exists $2 \leq r_1 < r$ such that $d_{r_1}^{mAdams} ([\tilde{z}]_{r_1}) \neq 0$.  We claim that there exists a nontrivial differential $d_{j}^{mAdams} ([u]_{j}) \neq 0$, where $u \in Ext^{s+b,t+b,\frac{t}{2}-a}_{\mathcal{A}_{*,*}^{\mathbb{C}}} (\mathbb{F}_p [\tau], \mathbb{F}_p[\tau])$ is non-$\tau$-divisible, $2 \leq j < r_1$, and $a > 0$.

We let $[v_1]_{r_1}$ denote $d_{r_1}^{mAdams} ([\tilde{z}]_{r_1}) \neq 0$.  By the commutativity of diagram \eqref{diagram: mASS and ASS}, we have $\phi_{r_1} ([v_1]_{r_1}) = d_{r_1}^{Adams} ([z]_{r_1}) = 0$.  So $[v_1]_{r_1}$ is $\tau$-torsion.  Let $a_1 > 0$ be the smallest integer such that $\tau^{a_1} [v_1]_{r_1} = 0$.  Then there exists a nontrivial differential $d_{r_2}^{mAdams} ([u_1]_{r_2}) = \tau^{a_1} [v_1]_{r_2} \neq 0$ with $2 \leq r_2 < r_1$.  The element $[u_1]_{r_2}$ is not divisible by $\tau$ on the $E_{r_2}$-page. Otherwise, we have 
$d_{r_2}^{mAdams} ([u_1]_{r_2}/\tau) = \tau^{a_1-1} [v_1]_{r_2}$, which contradicts the definition of $a_1$.

By Lemma \ref{lem: compute motivic weight}, we have $[\tilde{z}]_{r_1} \in E_{r_1}^{s+k,t+k,\frac{t}{2}} (S)$.  This implies 
\begin{equation*}
    [v_1]_{r_1} = d_{r_1}^{mAdams} ([\tilde{z}]_{r_1}) \in E_{r_1}^{s+k+r_1,t+k+r_1-1,\frac{t}{2}} (S).
\end{equation*}
Comparing the degrees, we have 
\begin{equation*}
    \tau^{a_1} [v_1]_{r_2} \in E_{r_2}^{s+k+r_1,t+k+r_1-1,\frac{t}{2}-a_1} (S),~ [u_1]_{r_2} \in E_{r_2}^{s+k+r_1-r_2,t+k+r_1-r_2,\frac{t}{2}-a_1} (S).
\end{equation*}

If $u_1$ is non-$\tau$-divisible, we can take $d_{r_2}^{mAdams} ([u_1]_{r_2}) \neq 0$ as the claimed differential.

Otherwise, since $[u_1]_{r_2}$ is not divisible by $\tau$ on the $E_{r_2}$-page, we conclude $u_1/ \tau$ does not lift to the $E_{r_2}$-page.  Then there exists differential $d_{r_3}^{mAdams} ([u_1/ \tau]_{r_3}) = [v_2]_{r_3} \neq 0$ with $2 \leq r_3 < r_2$.  

By the commutativity of diagram \eqref{diagram: mASS and ASS}, we have 
\begin{equation*}
    \phi_{r_3} ([v_2]_{r_3}) = d_{r_3}^{Adams} (\phi_{r_3} ([u_1/ \tau]_{r_3})) = d_{r_3}^{Adams} (\phi_{r_3} ([u_1]_{r_3})) = \phi_{r_3} (d_{r_3}^{mAdams} ([u_1]_{r_3})),
\end{equation*}
and $d_{r_3}^{mAdams} ([u_1]_{r_3}) = 0$ since $u_1$ can be lifted to the $E_{r_2}$-page. 
So $[v_2]_{r_3}$ is $\tau$-torsion.  
Let $a_2 > 0$ be the smallest integer such that $\tau^{a_2} [v_2]_{r_3} = 0$.  Then there exists a nontrivial differential $d_{r_4}^{mAdams} ([u_2]_{r_4}) = \tau^{a_2} [v_2]_{r_4} \neq 0$ with $2 \leq r_4 < r_3$.  The element $[u_2]_{r_4}$ is not divisible by $\tau$ on the $E_{r_4}$-page. 

For the degrees, we have
\begin{align*}
		[u_1/ \tau]_{r_3} & \in  E_{r_3}^{s+k+r_1-r_2,t+k+r_1-r_2,\frac{t}{2}-(a_1-1)} (S), \\
		[v_2]_{r_3} & \in  E_{r_3}^{s+k+r_1-r_2+r_3,t+k+r_1-r_2+r_3-1,\frac{t}{2}-(a_1-1)} (S), \\
		\tau^{a_2} [v_2]_{r_4} & \in  E_{r_4}^{s+k+r_1-r_2+r_3,t+k+r_1-r_2+r_3-1,\frac{t}{2}-(a_1-1)-a_2} (S), \\
		[u_2]_{r_4} & \in  E_{r_4}^{s+k+r_1-r_2+r_3-r_4,t+k+r_1-r_2+r_3-r_4,\frac{t}{2}-(a_1-1)-a_2} (S). \\
\end{align*}
If $u_2$ is non-$\tau$-divisible, we can take $d_{r_4}^{mAdams} ([u_2]_{r_4}) \neq 0$ as the claimed differential.  Otherwise, we can repeat the process and obtain $u_3, u_4, \cdots$.  Note there are only finitely many integers between $2$ and $r_1$.  After repeating this process several times, we will eventually obtain a desired nontrivial differential $d_{j}^{mAdams} ([u]_{j}) \neq 0$, where $u \in Ext^{s+b,t+b,\frac{t}{2}-a}_{\mathcal{A}_{*,*}^{\mathbb{C}}} (\mathbb{F}_p [\tau], \mathbb{F}_p[\tau])$ is non-$\tau$-divisible, $2 \leq j < r_1$, and $a > 0$.

By Proposition \ref{prop: motivic e2 page for sphere} and Lemma \ref{lem: direct sum decompose Adams E2}, we can write 
\begin{equation*}
    u = \sum_{i \in C_{s+b,t+b}} \tau^{n_i} \tilde{u}_i,
\end{equation*}
where $u_i$ can be detected by some 
$m_i \in Ext^{s+b-i,t+b-i}_{P_*}(\mathbb{F}_{p}, I^i / I^{i+1})$. Comparing the motivic weights using Lemma \ref{lem: compute motivic weight}, we get $n_i = \frac{b-i+2a}{2} \geq 0$ for $\tilde{u}_i \neq 0$.

Note $u$ is non-$\tau$-divisible. This implies the $n_i = 0$ term is nontrivial.
In particular, we have $b+2a \in C_{s+b,t+b}$.  This implies
\begin{equation}
   b+2a \leq s+b, \quad b+2a \equiv t+b ~ ~ \text{mod} ~ q.
\end{equation}
Note $x \in Ext^{s,t}_{P_*}(\mathbb{F}_p, I^k / I^{k+1})$ implies $t \equiv 0$ mod $q$.  In summary, we have 
\begin{equation}
    0 < 2a \leq s, \quad  2a \equiv 0  ~ ~ \text{mod} ~ q.
\end{equation}
Then $s \geq 2a \geq 2(p-1)$, this contradicts $s < 2p-2$.  Thus we have proved (ii).
\end{proof}

Now we proceed to prove Theorem \ref{thm: alg eq adams} and Theorem \ref{thm: nontrivial alg d then nontrivial adams d}.

\begin{proof}[Proof of Theorem \ref{thm: alg eq adams}]

Our strategy is to compare the differentials via diagram \eqref{diagram: three comparisons}.  As we will see, with the given assumptions, diagram \eqref{diagram: three comparisons} can be specialized to the following diagram.
\begin{align}
\label{diagram: summary diagram one}
\xymatrix{
 [x]_r  \ar[r]^{\kappa_r}  \ar[d]_{d_r^{alg}}  & [\tilde{z}]_r \ar[d]^{\delta_r^{mAdams}} & [\tilde{z}]_r \ar[l]_{\psi_r} \ar[d]^{d_r^{mAdams}} \ar[r]^{\phi_r}  & [z]_r \ar[d]^{d_r^{Adams}}  \\
 [y]_r  \ar[r]^{\kappa_r}   & [\tilde{w}]_r & [\tilde{w}]_r \ar[l]_{\psi_r} \ar[r]^{\phi_r} & [w]_r \\
} 
\end{align} 

By Lemma \ref{lem: lifting to Er}, $[\tilde{z}]_r$ is well defined and we have $\phi_r ([\tilde{z}]_r) = [z]_r$.  By Lemma \ref{lem: compute motivic weight}, we have $[\tilde{z}]_r \in E_r^{s+k,t+k,\frac{t}{2}} (S)$.  Then we can write $d_r^{mAdams} ([\tilde{z}]_r) = [u]_r$, where $u \in Ext^{s+k+r,t+k+r-1,\frac{t}{2}}_{\mathcal{A}_{*,*}^{\mathbb{C}}} (\mathbb{F}_p [\tau], \mathbb{F}_p[\tau])$. 
By Proposition \ref{prop: motivic e2 page for sphere} and Lemma \ref{lem: direct sum decompose Adams E2}, we can write 
\begin{equation}
\label{eq: decompose u}
    u = \sum_{i \in C_{s+k+r,t+k+r-1}} \tau^{n_i} \tilde{u}_i,
\end{equation}
where $u_i$ can be detected by some 
$m_i \in Ext^{s+k+r-i,t+k+r-1-i}_{P_*}(\mathbb{F}_{p}, I^i / I^{i+1})$. Comparing the motivic weights using Lemma \ref{lem: compute motivic weight}, we get $n_i = \frac{k+r-1-i}{2}$ for $\tilde{u}_i \neq 0$.

Since $n_i \geq 0$, this forces 
$i \leq k+r-1$.  By definition, $i \in C_{s+k+r,t+k+r-1}$ implies
\begin{equation}
    i \geq 0, \quad i \equiv t+k+r-1 \quad \text{mod} ~  q,
\end{equation}
where we denote $q = 2(p-1)$.  Moreover, since $x \in Ext^{s,t}_{P_*}(\mathbb{F}_p, I^k / I^{k+1})$, we have $t \equiv 0$ mod $q$.
In summary, we have
\begin{equation}
\label{eq: summary of i}
   0 \leq i \leq k+r-1 , \quad    i \equiv k+r-1 \quad \text{mod} ~ q
\end{equation}
for nontrivial $\tilde{u}_i$.

By assumption, we have $0< k+r-1 < q$.  Then \eqref{eq: summary of i} forces $i = k+r-1$, the corresponding $n_i = 0$.
So we can rewrite \eqref{eq: decompose u} as $u = \tilde{w}$, where we let $w$ denote $u_{k+r-1}$.  We also let $y$ denote $m_{k+r-1}$ detecting $w$.

We have $d_r^{mAdams} ([\tilde{z}]_r) = [\tilde{w}]_r$.  Note $\phi_r ([\tilde{w}]_r) = [w]_r$.  The commutativity of diagram \eqref{diagram: mASS and ASS} implies $d_r^{Adams} ([z]_r) = [w]_r$.

Note $[\tilde{z}]_r$ and $[\tilde{w}]_r$ are non-$\tau$-divisible, $\psi$ sends $[\tilde{z}]_r$ and $[\tilde{w}]_r$ to the corresponding elements in $E_r^{*,*,*} (C \tau)$ of the same form, which we abuse the notation and still denote by $[\tilde{z}]_r$ and $[\tilde{w}]_r$ respectively.  The commutativity of diagram \eqref{diagram: mASS and mASS} implies $\delta_r^{mAdams} ([\tilde{z}]_r) = [\tilde{w}]_r$. 

Finally, the the isomorphism $\kappa$ associates $\tilde{z}$ with $x$ and  $\tilde{w}$ with $y$.  Hence $\kappa_r ([x]_r) = [\tilde{z}]_r$,  $\kappa_r ([y]_r) = [\tilde{w}]_r$, and $d_r^{alg} ([x]_r) = [y]_r$. 

Now we have completed diagram \eqref{diagram: summary diagram one}.  The results of the theorem follow directly.

\end{proof}

\begin{proof}[Proof of Theorem \ref{thm: nontrivial alg d then nontrivial adams d}]

We can prove the following equivalent statement: suppose $r \geq 2$ is an integer such that $z$ can be lifted to the $E_r$-page (of the ASS). Then $x$ can also be lifted to the $E_r$-page (of the algNSS).

We will establish the statement using an inductive approach. By definition, $x$ can be lifted to the $E_2$-page. Next, assuming that $x$ can be lifted to the $E_i$-page with $2 \leq i < r$, we need to show that $d_i^{alg} ([x]_i) = 0$. This will imply that $x$ can be lifted to the $E_{i+1}$-page.

We study the differential $d_i^{alg} ([x]_i)$ via diagram \eqref{diagram: three comparisons}.  As we will see, with the given assumptions, diagram \eqref{diagram: three comparisons} can be specialized to the following diagram.
\begin{align}
\label{diagram: summary diagram two}
\xymatrix{
 [x]_i  \ar[r]^{\kappa_i}  \ar[d]_{d_i^{alg}}  & [\tilde{z}]_i \ar[d]^{\delta_i^{mAdams}} & [\tilde{z}]_i \ar[l]_{\psi_i} \ar[d]^{d_i^{mAdams}} \ar[r]^{\phi_i}  & [z]_i \ar[d]^{d_i^{Adams}}  \\
 0  \ar[r]^{\kappa_i}   & 0 & 0 \ar[l]_{\psi_i} \ar[r]^{\phi_i} & 0\\
} 
\end{align}

By Lemma \ref{lem: lifting to Er},  $\tilde{z}$ can be lifted to the $E_r$-page of the mASS of the sphere.    Hence $d_i^{mAdams} ([\tilde{z}]_i) = 0$. Note $[\tilde{z}]_i$ is non-$\tau$-divisible, so $\psi$ sends $[\tilde{z}]_i$ to the corresponding element in $E_r^{*,*,*} (C \tau)$ of the same form, which we abuse the notation and still denote by $[\tilde{z}]_i$.  The commutativity of diagram \eqref{diagram: mASS and mASS} implies $\delta_i^{mAdams} ([\tilde{z}]_i) = 0$. Finally, using the isomorphism $\kappa$, we deduce $d_i^{alg} ([x]_i) = 0$. 

\end{proof}

\bibliographystyle{plain}
\bibliography{adams}

\end{document}